\documentclass[12pt,a4paper]{article}
\usepackage[latin1]{inputenc}
\usepackage[english]{babel}
\usepackage{amsmath,amssymb,amsfonts,amsthm,color}
\usepackage{bussproofs}
\usepackage{hyperref}
\usepackage{pdfpages}
\makeatletter
\newcommand{\todo}[1]{\marginpar{\textbf{TODO\footnotemark}}\@latex@warning{TODO: #1}\footnotetext{ #1}}
\makeatother

\theoremstyle{plain}
\theoremstyle{definition}
\newtheorem{theorem}{Theorem}[section]
\newtheorem{lemma}[theorem]{Lemma}

\theoremstyle{definition}
\newtheorem{definition}[theorem]{Definition}

\theoremstyle{remark}

\newcommand{\m}{\ensuremath{\mathfrak{M}}}

\newcommand{\SFour}{\mathsf{S4}}
\newcommand{\iSFour}{\mathsf{iS4}}
\newcommand{\LP}{\mathsf{LP}}
\newcommand{\iLP}{\mathsf{iLP}}

\newcommand{\LJ}{\mathcal{L}_\mathsf{J}}
\newcommand{\Prop}{\mathsf{Prop}}
\newcommand{\Pow}{\mathcal{P}}
\newcommand{\justifies}{:}
\newcommand{\CS}{\mathsf{CS}}
\newcommand{\Tm}{\mathsf{Tm}}
\newcommand{\HJI}{\mathsf{iJT4}}
\newcommand{\HJICS}{{\HJI}_{\CS}}

\newcommand{\bang}{\mathop{!}\!}

\newcommand{\M}{\m}

\begin{document}

\title{Modular Models\\for Intuitionistic Justification Logic}
\author{Michel Marti \and Thomas Studer}

\maketitle

\begin{abstract}
We present the logic iJT4, which is an explicit version of intuitionistic S4 and establish soundness and completeness with respect to modular models.
\end{abstract}

\section{Introduction}

Justification logics are explicit modal logics in the sense that they 
unfold the $\Box$-modality in families of so-called justification terms.
Instead of formulas $\Box A$, meaning that $A$ is known, justification logics include formulas $t:A$, meaning that $A$ is known for reason $t$.

Artemov's original semantics for the first justification logic, the Logic of Proofs~$\LP$, was a provability semantics that interpreted $t:A$ roughly as \emph{$t$ represents a proof of $A$} in the sense of a formal proof predicate in Peano Arithmetic~\cite{Art95TR,Art01BSL,KSweak}.

Later Fitting~\cite{Fit05APAL} interpreted justifications as evidence in a more general sense and introduced epistemic, i.e., possible world, models for justification logics. These models have been further developed to modular models as we use them in this paper~\cite{Art12SL,KuzStu12AiML}. 
This general reading of justification led to many applications in epistemic logic~\cite{Art06TCS,Art08RSL,BalRenSme14APAL,
BucKuzStu11JANCL,BucKuzStu11WoLLIC,bks12a,2014arXiv1407.4647G,komaogst,KuzStu13LFCS}. 

Given the interpretation of $\LP$ in Peano Arithmetic, it was a natural question to find an intuitionistic version $\iLP$ of $\LP$ that is the logic of proofs of Heyting arithmetic. The work by Artemov and Iemhoff~\cite{ArtIem07JSL} and later by Dashkov~\cite{Das11JLC} provides such an $\iLP$. It turned out that $\iLP$ is not only $\LP$ with the underlying logic changed to intuitionistic propositional logic. In order to get a complete axiomatization with respect to provability semantics, one also has to include certain admissible rules of Heyting arithmetic as axioms in $\iLP$ so that they are represented by novel proof terms.

The main contribution of the present paper is that these additional axioms are not needed if we are interested in completeness with respect to modular models. 
We introduce the intuitionistic justification logic $\HJICS$, which is simply $\LP$ over an intuitionistic base instead of a classical one but without any additional axioms. We introduce possible world models for $\HJICS$ that are inspired by the Kripke semantics for intuitionistic $\SFour$ and establish completeness of $\HJICS$ with respect to these models.

\section{Intuitionistic Justification Logic}

In this section, we introduce the syntax for the justification logic $\HJICS$, which is the explicit analogue of the intuitionistic modal logic $\iSFour$.

\begin{definition}[Justification Terms]
We assume a countable set of justification constants and a countable set of justification variables. Justification terms are inductively defined by:
\begin{enumerate}
 \item each justification constant and each justification variable is a justification term;
 \item  if $s$ and $t$ are justification terms, then so are 
 \begin{itemize}
  \item $(s \cdot t)$, read $s$ dot $t$,
  \item $(s + t)$, read $s$ plus $t$,
  \item $\bang s$, read bang $s$.
 \end{itemize}
\end{enumerate}
We denote the set of terms by $\Tm$.
\end{definition}

\begin{definition}[Formulas]
We assume a countable set~$\Prop$ of atomic propositions.
The set of formulas~$\LJ$ is inductively defined by:
\begin{enumerate}
 \item every atomic proposition is a formula;
 \item the constant symbol $\bot$ is a formula;
 \item If $A$ and $B$ are formulas, then $(A \land B)$, $(A \lor B)$ and $(A \rightarrow B)$ are formulas;
 \item if $A$ is a formula and $t$ a term, then $t:A$  is a formula.
\end{enumerate} 
\end{definition}

\begin{definition}
The axioms of $\HJI$ consist of the following axioms:
\begin{enumerate}
\item all axioms for intuitionistic propositional logic
\item $t \justifies (A \rightarrow B) \rightarrow ( s \justifies A \rightarrow t\cdot s \justifies B)$
\item $t \justifies A \rightarrow t+s \justifies A$ and $s \justifies A \rightarrow t+s \justifies A$
\item $t\justifies A \to A$
\item $t\justifies A \to \bang t \justifies t \justifies A$
\end{enumerate}
A {\em constant specification} $\CS$ is any subset
\[
\CS \subseteq \{ (c,A)  \ |\ \text{$c$ is a constant and $A$ is an axiom of $\HJI$} \}.
\]
A constant specification $\CS$ is called
{\em axiomatically appropriate} if for each axiom $A$ of $\HJI$, there is a constant $c$ such that $(c,  A) \in \CS$.

For a constant specification $\CS$ the deductive system $\HJICS$ is the Hilbert system given by the axioms above  and 
by the rules modus ponens and axiom necessitation:
  \[
  \begin{array}{c}
    A \quad A\rightarrow B
    \\ \hline
    B
  \end{array}
  \qquad
  \begin{array}{c}
    (c, A) \in \CS
    \\ \hline
    c \justifies A
  \end{array}
  \]
\end{definition}

As usual in justification logic, we can establish the deduction theorem and
the internalization property.

\begin{theorem}[Deduction Theorem] \label{Deduction Theorem}
For every set of formulas $M$ and all formulas $A,B$ we have that
\[
 M \cup \{ A \} \vdash_{\HJICS} B \quad \Longleftrightarrow \quad M \vdash_{\HJICS} A \to B.
\]
\end{theorem}

\begin{lemma}[Internalization for Arbitrary Terms] \label{Internalization for Arbitrary Terms}
Let $\CS$ be an axiomatically appropriate constant specification. For arbitrary formulas $A, B_1 , \dots , B_n$ and arbitrary terms $s_1 , \dots , s_n$, if
\[
 B_1 ,\dots , B_ n \vdash_{\HJICS} A,
\]
then there is a term $t \in \Tm$ such that
\[
 s_1 : B_1 , \dots , s_n :B_n \vdash_{\HJICS} t:A.
\]
\end{lemma}

\section{Basic Modular Models}

Basic modular models are syntactic models for justification logic.
Yet, our basic modular models will include possible worlds in order to deal with the intuitionistic base logic.
After defining basic modular models for intuitionistic justification logic, we will prove soundness and completeness.

\begin{definition}[Basic evaluation]
 A  \emph{basic evaluation} is a tuple $(W, \leq, *)$ where
 \[
  W \neq \varnothing  \text{ and $\leq$ is a partial order on } W,
 \]
 \[
  * : \Prop \times W \to \{0,1 \} \quad \quad * : \Tm \times W \to \Pow(\LJ)
 \]
(where we often write $t^*_w$ for $*(t,w)$),
such that for arbitrary $s, t \in \Tm$ and any formula $A$,
   \begin{itemize}
    \item[(1)] $s^*_w  \cdot t^*_w \subseteq (s \cdot t)^*$;
    \item[(2)] $s^*_w  \cup t^*_w  \subseteq (s + t)^*_w$;
    \item[(3)] $(t,A) \in \CS \quad \Longrightarrow \quad A \in t^*_w$;
    \item[(4)] $s:s^*_w \subseteq (\bang s)^*_w $.
   \end{itemize}
 Furthermore, it has to satisfy the following monotonicity conditions:
\begin{itemize}
 \item[(M1)] $  p^*_w = 1 \text{ and } w \leq v \quad \Longrightarrow \quad
p^*_v = 1$;
 \item[(M2)] $  w \leq v \quad \Longrightarrow \quad t^*_w \subseteq p^*_v$.
\end{itemize}
\end{definition}
Strictly speaking we should use the notion of a $\CS$ basic evaluation because of condition (3) depends on a  given $\CS$.
However, the constant specification will always be clear from the context and we can safely omit it. The same also holds for modular models (to be introduced later).

\begin{definition}[Truth under Basic Evaluation]
Let $\M = (W, \leq, *)$ be a basic evaluation. 
For $w \in W$,  we define $(\M,w) \vDash A$ by induction on the formula $A$ as follows:
 \begin{itemize}
 \item $(\M,w)  \nvDash \bot$;
 \item $(\M,w)  \vDash p$ iff $*(p,w) = 1$;
 \item $(\M,w)  \vDash A \land B$ iff $(\M,w)  \vDash A$ and $(\M,w)  \vDash B$;
 \item $(\M,w)  \vDash A \lor B$ iff $(\M,w)  \vDash A$ or $(\M,w)  \vDash B$;
 \item $(\M,w)  \vDash A \rightarrow B$ iff $(\M,v) \vDash B$ for all $v \geq w$ with $(\M,v) \vDash A$;
 \item $(\M,w)  \vDash t:A$ iff $A \in t^*_w$.
 \end{itemize}
\end{definition}

\begin{lemma}[Monotonicity]
For any basic evaluation $\M = (W, \leq, *)$, states $w, v \in W$ and any formula $A$:

\[
 (\M,w) \vDash A \text{ and } w \leq v   \quad \Longrightarrow \quad (\M,v) \vDash A.
\]
\end{lemma}

\begin{definition}[Factive Evaluation]
A basic evaluation $\M = (W, \leq, * )$ is called \emph{factive} iff 
\[
 A \in t^*_w \quad \Longrightarrow \quad (\M,w) \vDash A
\]
for all formulas $A$, all justification terms $t$ and all states $w \in W$.
\end{definition}

\begin{definition}[Basic modular model]
A \emph{basic modular model} is a basic evaluation $(W, \leq, *)$ that is factive.

We say that a formula $A$ is \emph{valid with respect to basic modular models} (in symbols $\vDash_\mathrm{basic modular} A$)
if for any basic modular model  $\M = (W, \leq, * )$ and any $w \in W$ we have $ (\M,w) \vDash A$.
\end{definition}

\begin{lemma}[Soundness of $\HJICS$ with respect to basic modular models]
For every formula $A$:
\[
 \vdash A \quad \Longrightarrow \quad \vDash_\mathrm{basic modular} A
\]

\end{lemma}

In order to show completeness, we need some auxiliary definitions and lemmas.

\begin{definition}
We call a set of formulas $\Delta$ \emph{prime} iff it satisfies the following conditions:

\begin{itemize}
 \item[(i)] $\Delta$ has the disjunction property, i.e., $A \lor B \in \Delta \Longrightarrow A \in \Delta$ or $B \in \Delta$;
 
 \item[(ii)] $\Delta$ is deductively closed, i.e., for any formula $A$, if $\Delta \vdash A$, then $A \in \Delta$; 
 
 \item[(iii)] $\Delta$ is consistent, i.e., $\bot \notin \Delta$.
\end{itemize}

From now on, we will use $\Sigma, \Delta, \Gamma$ for prime sets of formulas.

\end{definition}

\begin{lemma}
 Let $N$ be an arbitrary set of formulas and let $A, B$ and $C$ be formulas.
 If 
 \[
  N \cup \{A \lor B \} \nvdash C \text{, then } N \cup \{A \} \nvdash C \text{ or } N \cup \{B \} \nvdash C.
 \]
\end{lemma}

\begin{proof}
By contraposition. Assume that
\[
 N \cup \{A \} \vdash C \text{ and } N \cup \{B \} \vdash C
\]
Then there are finite subsets $N_1 \subseteq N \cup \{ A \}$ and $N_2 \subseteq N \cup \{ B\}$ such that
\[
 \vdash \bigwedge N_1  \to C \text{ and } \vdash \bigwedge N_2  \to C
\]

Now let $N_1' := N_1 \setminus \{ A \}$ and  $N_2' := N_2 \setminus \{ B \}$. Then $N_1', N_2'$ are finite subsets of $N$, and

\[
 \vdash \bigwedge (N_1' \cup \{A\} ) \to C \text{ and } \vdash \bigwedge (N_2' \cup \{B\} ) \to C.
\]
So 
 \[
  \vdash \bigwedge N_1'    \rightarrow ( A \rightarrow C) \text{ and } \vdash \bigwedge N_2'  \rightarrow ( B \rightarrow C).
 \]
Strengthening the antecedent, we get
 \[
  \vdash \bigwedge (N_1' \cup N_2') \to ( A \rightarrow C) \text{ and } \vdash \bigwedge (N_1' \cup N_2')  \to ( B \rightarrow C))
 \] 
and, therefore,
 \[
  \vdash \bigwedge (N_1' \cup N_2') \to (( A \rightarrow C) \land ( B \rightarrow C)).
 \]
 By propositional reasoning we get
 
  \[
   \vdash \bigwedge (N_1' \cup N_2') \to ( (A \lor B) \rightarrow C),
  \]
  which means that
  \[
  \vdash \bigwedge (N_1' \cup N_2'  \cup \{ A \lor B\}  ) \rightarrow  C.
  \]
  Since $N_1'$ and $N_2'$ are finite subsets of $N$, $(N_1' \cup N_2'  \cup \{ A \lor B\}$ is a finite subset of $N \cup \{  A \lor B \}$,
  so by definition
  \[
      N \cup \{ A \lor B\} \vdash  C.
	\qedhere
  \]
\end{proof}

\begin{theorem}[Prime Lemma] \label{Prime Lemma}
 Let $B$ be a formula and let $N$ be a set of formulas such that $N \nvdash B$.
 Then there exists a prime set $\Pi$ with $N \subseteq \Pi$ and $\Pi \nvdash B$.
\end{theorem}

\begin{proof}
 
Let $(A_n)_{n \in \mathbb{N}}$ be an enumeration of all formulas.

Now we define $N_0 := N$,
\[
N_{i+1} := \left\{
  \begin{array}{ll}
     N_i \cup \{A_i \},  &\text{if } N_i \cup \{A_i \} \nvdash B\\
    N_i, &\text{ otherwise}
  \end{array}
\right.
\]
and finally 
\[
 N^\star := \bigcup_{i \in \mathbb{N}} N_i
\]
By induction in $i$, one can easily show that for all $i \in \mathbb{N}: N_i \nvdash B$ and, therefore, $N^\star \nvdash B$.

It remains to show that $N^\star$ is prime. We have the following: 
\begin{itemize}
 \item $\bot \notin  N^\star$: Since $\bot \notin N_i$ for all $i \in \mathbb{N}$, which can be shown by induction on $i$.
 \item $ N^\star$ is deductively closed: Assume it is not, i.e., there is a formula $A$ with
 \[
   N^\star \vdash A \text{ but } A \notin  N^\star
 \]

 Since $N^\star \vdash A$ but $N^\star \nvdash B$, we know that
 
 \[
   N^\star \cup \{A \} \nvdash B \quad
 \]
 Otherwise, by the deduction theorem \ref{Deduction Theorem}
 \[
   N^\star \vdash A \rightarrow B \text{ and }  N^\star \vdash A 
 \]
so by propositional reasoning,
\[
  N^\star \vdash B, \text{ which contradicts our observation above.}
\]

\medskip
 Since $(A_n)_{n \in \mathbb{N}}$ is an enumeration of all formulas, there is some $i$ such that $A = A_i$.
 But then
 \[
  N_i \cup \{A_i \} \nvdash B.
 \]
So by construction
 \[
  N_{i + 1} = N_i \cup \{ A_i \}
 \]
 and, therefore,
 \[
  A = A_i \in N_{i +1} \subseteq  N^\star,
 \]
which contradicts our assumption.

 \item  $N^\star$ has the disjunction property:
 Assume that $C \lor D \in  N^\star$.
 Then there is some $i$ such that $C \lor D = A_i$ and there are $i_1, i_2$ such that
 \[
  C = A_{i_1} \text{ and } D =  A_{i_2}
 \]
 Now we have
 \[
  N^\star =  N^\star\cup \{C \lor D \} \nvdash B
 \]
 By the lemma above it follows that
\[
 N^\star \cup \{C \} \nvdash B \text { or } N^\star \cup \{D \} \nvdash B
\]

In the first case, we have that 
\[
 N_{i_1} \cup \{A_{i_1} \} \nvdash B
\]
so by the definition of $N_{i_1 +1}$,
\[
 N_{i_1 +1} = N_{i_1} \cup \{A_{i_1} \} = N_{i_1} \cup \{ C \}
\]
which means that $C \in N_{i_1 +1 }$ and therefore $C \in N^\star$.
The second case is analogous.
\qedhere
\end{itemize}
\end{proof}

\begin{lemma}\label{A_previous_lemma}
 Let $\Delta$ be a prime set and $t$ be a justification term. Then
 \[
  t^{-1} \Delta \subseteq \Delta.
 \]
\end{lemma}

\begin{proof}
Let $A \in t^{-1} \Delta$. Then $t:A \in \Delta$. Since $\Delta$ is deductively closed, it contains all axioms, thus
$t:A \rightarrow A \in \Delta$. Again, since $\Delta$ is deductively closed, it follows by $(MP)$ that $A \in \Delta$.
\end{proof}

\begin{definition}[Canonical Basic Modular Model]

The canonical basic modular model is 
\[
 B^{can} := (W^{can}, \leq^{can}, *^{can}) 
\]
where
 \begin{itemize}
 \item[(i)] $W^{can} \quad := \quad \{ \Delta \subseteq \LJ : \Delta \text{ is prime} \}$
 
 \item[(ii)] $ \leq^{can} \quad := \quad \subseteq$
 
 \item[(iii)] $*^{can}(p, \Delta) = 1$ iff $P \in \Delta$

 \item[(iv)] $*^{can}(t, \Delta) := t^{-1} \Delta := \{ A \mid t:A \in \Delta \} $

\end{itemize}
\end{definition}

\begin{lemma}
$B^{can}$ is a basic evaluation.
\end{lemma}
\begin{proof}
 $W \neq \varnothing$: By the consistency of $\HJICS$ we have that $\varnothing \nvdash \bot$, it follows by the prime lemma \ref{Prime Lemma} that there exists a prime set, so $W^{can} \neq \varnothing$. 
 
 Next, we check the conditons on the sets of formulas $t^{*^{can}}_w$.
 \begin{itemize}
   \item[(1)] 
	$s^{*^{can}}_w \cdot t^{*^{can}}_w \subseteq (s \cdot t)^*_w$.
    Let $A \in s^{*^{can}}_w  \cdot t^{*^{can}}_w$. Then there is a formula $B \in t^{*^{can}}_w$ such that $B \to A \in s^{*^{can}}_w$.
    So $s:B \to A \in w$ and $t:B \in w$.
    Since w is a prime set, it is deductively closed, so it contains the axiom $s: (B \to A) \to (t:B \to s \cdot t: A)$. Again since $w$ is deductively closed, it follows by (MP) that
    $s \cdot t: A \in w$, so $A \in (s \cdot t)^{-1} w = (s \cdot t)^{*^{can}}_w$.
   \item[(2)] 
	$s^{*^{can}}_w  \cup t^{*^{can}}_w  \subseteq (s + t)^{*^{can}}_w$.
   Let $A \in s^{*^{can}}_w \cup t^{*^{can}}_w$. Case 1: $A \in s^{*^{can}}_w = s^{-1} w$. Then $s:A \in w$. Since $w$ is deductively closed, it contains the  axiom 
   $s:A \to (s+t):A$. Thus by (MP) we find $(s+t):A \in w$, i.e., $A \in (s+t)^{-1}w = (s + t)^{*^{can}}_w$. The second case is analogous.
   \item[(3)] 
	$(t,A) \in CS \quad \Longrightarrow \quad A \in t^{*^{can}}_w$.
   By axiom necessitation we have that $\HJICS \vdash t:A$, so $w \vdash t:A$.
   Since $w$ is deductively closed, it follows  that $t:A \in w$, so $A \in t^{-1} w = t^{*^{can}}_w$. 
   \item[(4)] 
	$s:s^{*^{can}}_w \subseteq (\bang s)^{*^{can}}_w $.
   Let $A \in s:s^{*^{can}}_w$. Then $A$ is of the form $s:B$ for some formula $B \in s^{*^{can}}_w = s^{-1} w$, i.e., $s:B \in w$.
   We find that the axiom
   $(s:B) \to~ \bang s:(s:B) \in w$, so $\bang s:(s:B) \in w$, which means that $s:B \in (\bang s)^{-1}w = (\bang s)^{*^{can}}_w $.
   \end{itemize}
Now we check the monotonicity conditions.
 \begin{itemize}
  \item[(M1)] Assume that $*(p, \Gamma) = 1$ and $\Gamma \subseteq \Delta$. By the definition of  of $*$ we have that $p \in \Gamma$, so $p \in \Delta$ hence $*(p, \Delta) = 1$.
  \item[(M2)] Now assume that $\Gamma \subseteq \Delta$. Then $t^{-1} \Gamma \subseteq t^{-1} \Delta$ which means $t^*_\Gamma \subseteq t^*_\Delta$.
\qedhere
  \end{itemize}
\end{proof}

\begin{lemma}[Truth Lemma] \label{Truth Lemma for the Canonical Basic Modular Model}
For any formula $A$ and any prime set $\Delta$ :

\[
 A \in \Delta \quad \Longleftrightarrow  \quad (*^{can},\Delta) \vDash A
\]
\end{lemma}
\begin{proof}
 By induction on the formula $A$. We distinguish the following cases.
 \begin{enumerate}
\item
 $A = p$ or $A = \bot$. By definition.

 \item
$A = B \land C$. 
Assume that $B \land C \in \Delta$. Since $\Delta$ is deductively closed, we have $B \in \Delta$ and $C \in \Delta$, so it follows by the induction hypothesis that $(*^{can},\Delta) \vDash B$ and $(*^{can},\Delta) \vDash C$.

For the other direction assume that $(*^{can},\Delta) \vDash B \land C$, so $(*^{can},\Delta) \vDash B$ and $(*^{can},\Delta) \vDash C$. By the induction hypothesis, we get that $B \in \Delta$ and $C \in \Delta$. Since $\Delta$ is deductively closed,
it follows that $B \land C \in \Delta$.
\item
$A = B \lor C$. 
Assume that $B \lor C \in \Delta$. Since $\Delta$ has the disjunction property, it follows that $B \in \Delta$ or $C \in \Delta$, so by the induction hypothesis, $(*^{can},\Delta) \vDash B$ or $(*^{can},\Delta) \vDash C$, so
$(*^{can},\Delta) \vDash B \lor C$.

For the other direction assume that $(*^{can},\Delta) \vDash B \lor C$. Then
\[
(*^{can},\Delta) \vDash B \text{ or } (*^{can},\Delta) \vDash C,
\]
so by the induction hypothesis, $B \in \Delta$ or $C \in \Delta$. 
Since $\Delta$ is deductively closed, it follows that $B \lor C \in \Delta$.
 \item
$A = B \to C$. 
Assume that $B \to C \in \Delta$. We have to show that $(*^{can},\Delta) \vDash B \to C$, so
let $\Gamma$ be a prime set such that  $\Delta \subseteq \Gamma$ and $(*^{can}, \Gamma) \vDash B$. It follows by the induction hypothesis that
$B \in \Gamma$, and since $B \to C \in \Gamma$ and $\Gamma$ is deductively closed, we have that $C \in \Gamma$. Applying the induction hypothesis again,
we get that $(*^{can}, \Gamma) \vDash C$. 

For the other direction assume that  $(*^{can},\Delta) \vDash B \to C$. We have to show that $B \to C \in \Delta$. Assume for a contradiction that
$B \to C \notin \Delta$. Since $\Delta$ is deductively closed, it follows that $\Delta \nvdash B \to C$.
It follows by the deduction theorem \ref{Deduction Theorem} that $\Delta \cup \{ B\} \nvdash C$.
By the prime lemma \ref{Prime Lemma}, there is a prime set $\Gamma$ such that $\Delta \cup \{ B \} \subseteq \Gamma$ and $\Gamma \nvdash C$, so in particular, $C \notin \Gamma$.
By the induction hypothesis it follows that $(*^{can}, \Gamma) \vDash B$ and $(*^{can}, \Gamma) \nvDash C$, contradicting our assumption that
$(*^{can},\Delta) \vDash B \to C$. 
\item
$A = t:B$. We have
\[
t:B \in \Delta
\quad\Longleftrightarrow\quad
B \in t^{-1} \Delta = *^{can}(t, \Delta)
\quad\Longleftrightarrow\quad
(*^{can}, \Delta) \vDash t:B
.
\]
\qedhere
\end{enumerate}
\end{proof}

\begin{lemma}
$B^{can}$ is a basic modular model.
\end{lemma}
\begin{proof}
We only have to show factivity, for which we use the truth lemma.
 Assume that 
\[
A \in *^{can}(t, \Delta) = t^{-1} \Delta.
\]
By Lemma~\ref{A_previous_lemma}
we know that $t^{-1} \Delta \subseteq \Delta$, so we have $A \in \Delta$.
 By the truth lemma for the canonical basic modular model, 
we can conclude that $(*^{can}, \Delta) \vDash A$. So factivity is shown.
\end{proof}

\begin{theorem}[Completeness of $\HJICS$ with respect to basic modular models]
For any formula $A$:
\[
 \vDash_\mathrm{basic modular} A \quad \Longrightarrow  \quad  \HJICS \vdash A
\]
 \end{theorem}

\begin{proof}
By contraposition. Assume that $\HJICS \nvdash A$.
By the prime lemma \ref{Prime Lemma}, there exists a prime set $\Delta$ such that $\Delta \nvdash A$.
In particular, $A \notin \Delta$. By the truth lemma \ref{Truth Lemma for the Canonical Basic Modular Model}, it follows that

\[
 (*^{can}, \Delta) \nvDash A
\]
since this structure is a basic modular model, it follows that
 \[
    \nvDash_\mathrm{basic modular} A.
\qedhere
 \]
\end{proof}


\end{document}